\documentclass[a4paper,12pt]{amsart}

\usepackage{amssymb,xspace,amscd,yhmath,mathdots,txfonts,
mathrsfs,yfonts}
 \DeclareMathAlphabet{\mathcalligra}{T1}{calligra}{m}{n}
\usepackage[matrix,arrow,curve]{xy}\CompileMatrices
\DeclareMathAlphabet{\mathpzc}{OT1}{pzc}{m}{it}
\usepackage{hyperref}

\numberwithin{equation}{section}

\theoremstyle{plain}
\newtheorem{thm}{Theorem} 

\newtheorem{cor}[thm]{Corollary}

\theoremstyle{definition}

\date{\today}

\title{On the nonvanishing of abstract
Cauchy-Riemann cohomology groups}
\author[J.~Brinkschulte]{Judith Brinkschulte}
\address{J.\ Brinkschulte:
Mathematisches Institut\\ Universit\"at Leipzig\\ 
Augustusplatz 10/11\\ 04109 Leipzig (Germany)}
\email{brinkschulte@math.uni-leipzig.de}

\author[C.D.~Hill]{C.Denson Hill}
\address{C.D.\ Hill:
Department of Mathematics\\ Stony Brook University
\\ Stony Brook, N.Y. 11794\\ USA}
\email{dhill@math.stonybrook.edu}

\author[M.~Nacinovich]{Mauro Nacinovich}
\address{M.\ Nacinovich:
Dipartimento di Matematica\\ II Universit\`a di Roma
``Tor Ver\-ga\-ta''\\ Via della Ricerca Scientifica\\ 00133 Roma
(Italy)}
\email{nacinovi@axp.mat.uniroma2.it}

\begin{document}

\newcommand{\pa}{\partial}
\newcommand{\opa}{\overline\pa}
\newcommand{\ol}{\overline }
\newcommand\Hdis{H_{\text{distr}}}
\numberwithin{equation}{section}
\newcommand\Ci{\mathcal{C}^{\infty}}
\newcommand\Dis{\mathcal{D}'}
\newcommand\Di{\mathcal{D}}
\newcommand\C{\mathbb{C}}  
\newcommand\R{\mathbb{R}}
\newcommand\Z{\mathbb{Z}}
\newcommand\N{\mathbb{N}}
\newcommand\PP{\mathbb{P}}
\newcommand\im{\mathrm{Im}\,}


\begin{abstract} In this paper we prove  infinite dimensionality of 
some local and global
 cohomology groups on abstract Cauchy-Riemann manifolds.
\end{abstract}

\maketitle

\section{Introduction}
It is natural to define abstractly smooth $CR$ structures on a smooth
manifold $M$. The motivation for this comes from the fact that, when $M$ is
embedded generically in an ambient complex manifold $X$, the complex
structure in $X$ induces a \textsl{tangential} $CR$ structure on $M$. If $X$ has complex
dimension $n+k$ and $M$ has real codimension $k$ in $X$, then $M$ has $CR$ dimension
$n$ and CR codimension $k$.
\par 
   For such abstract $CR$ manifolds, one can also define the \textsl{tangential}
$\bar{\partial}_M$ complex and the associated global abstract cohomology groups
$H^{p,q}(M)$, which are the analogues of the Dolbeault cohomology groups.
In spite of the rather large literature concerning $CR$ manifolds, these
abstract $CR$ cohomology groups remain somewhat mysterious. In this paper
we show that some of these global cohomology groups must be infinite
dimensional, or non Hausdorff, whenever one has a certain condition on
the Levi form of the $CR$ structure. What makes these results curious is
that the required condition on the Levi form needs to be satisfied
only at a single (micro-local) point on M; yet the conclusion is global. 
\par 
  The circle of ideas surrounding these results began with the famous paper \cite{L57},
 where Lewy found an example of a complex vector field $L$, in three real
variables, with real analytic coefficients, such that the
inhomogeneous equation $Lu = f$ had no local solutions $u$ for almost all
prescribed $f\in\mathcal{C}^{\infty}(\mathbb{R}^3)$. Lewy's argument was based on
the Schwarz reflection principle. A short time later H{\"o}rmander generalized
this example to allow any number of real variables, and
also higher order linear operators $L$, first for real analytic coefficients \cite{Ho60a},
and then for smooth coefficients \cite{Ho60}. 
H{\"o}rmander employed an
entirely different type of argument, which used functional analysis
to obtain some a priori estimates, and then showed how to construct
\textit{peak} functions $f$ that violated the assumed estimates. This approach
is reminiscent of the Sommerfeld radiation condition. In both works the
linear independence of the Lie bracket $[L,\bar{L}]$ from $L$ and $\bar{L}$ played
a crucial role. However these results were all for the scalar case of
one PDE for one unknown function.
\par 
   Somewhat later in \cite{AH1, AH2}   another generalization of Lewy's
example was found. The scalar equation was replaced by a system of
PDE's which correspond to $\bar{\partial}_M{ u} = f$, where $\bar{\partial}_M$ is the operator
in the tangential $CR$ complex associated to a real hypersurface $M$
embedded in $\mathbb{C}^N$. The condition on the Lie bracket was replaced by
an assumption on the Levi form of the surface $M$. There it was shown
how the signature of the Levi form is related to the places in the
$\bar{\partial}_M$-complex where one has a local nonsolvability result analogous
to that of Lewy and H{\"o}rmander. The arguments employed there were still
different, being of a geometrical, as well as of an analytical nature.
Also now $f$ has to satisfy compatibility conditions $\bar{\partial}_Mf = 0$,
which bring in additional complications. Subsequently in \cite{AFN, HN2} 
these local nonsolvability results were generalized to the situation
where $M$ has higher codimension. In \cite{AFN, HN2} 
the approach was to go back to
H{\"o}rmander's proof for the scalar case; they managed to construct
analogous \textit{peak} forms, using the embedding of $M$ into a complex manifold. 
H{\"o}rmander's technique was further pushed  to general 
overdetermined systems in \cite{N}.
\par
   All of the nonsolvability results mentioned above are local; that
is, the entire discussion is taking place in an arbitrarily small
neighborhood of some point $p_0$. We call this the \textit{failure of the
Poincare lemma for $\bar{\partial}_M$ at ~$p_0$.}
\par 
It is not obvious that this failure entails the infinite dimensionality of some
global cohomology groups on $M$. But this is in fact what we found: 
the signature at a point of a scalar Levi form, which is a
micro local condition, yields the infinite dimensionality of some
\textit{global} cohomology groups.  \par
For \textit{compact} $M$, some of our results have been obtained previously in \cite{BH}.

\section{Definitions}

 We consider a $\mathcal{C}^\infty$ smooth connected abstract $CR$ manifold of type $(n,k)$. 
Here an abstract $CR$ manifold of type $(n,k)$ is a triple $(M, HM, J)$, where $M$ is a paracompact smooth differentiable real manifold of dimension $2n+k$, $HM$ is a subbundle of rank $2n$ of the tangent bundle $TM$, and $J: HM \rightarrow HM$ is a smooth fiber preserving bundle isomorphism with $J^2= -\mathrm{Id}$. We also require that $J$ be formally integrable; i.e. that we have
$$\lbrack T^{0,1}M,T^{0,1}M\rbrack \subset T^{0,1}M$$
where 
$$ T^{0,1}M = \lbrace X+ iJX\mid X\in \Gamma(M,HM)\rbrace \subset \Gamma(M,\mathbb{C}TM),$$
with $\Gamma$ denoting smooth sections.

The $CR$ dimension of $M$ is $n\geq 1$ and the $CR$ codimension is $k\geq 1$.\\

We denote by $\opa_M$ the tangential Cauchy-Riemann operator on $M$ acting on smooth $(p,q)$-forms $f\in\mathcal{C}^\infty_{p,q}(M)$. The associated cohomology groups of $\opa_M$ acting on smooth forms will be denoted by $H^{p,q}(M)$, $0\leq p\leq n+k,\ 0\leq q\leq n$. For more details on the $\opa_M$ complex, we refer the reader to \cite{HN1} or 
\cite{HN2}.\\

In the present paper, $M$ is allowed to be compact, without boundary, but our main interest is the case where $M$ is a noncompact (open) manifold. When $M$ is compact, the infinite dimensionality of the appropriate cohomology 
groups, defined using smooth forms, 
was proved  in \cite{BH}.\\

Note that  the spaces $C^\infty_{p,q}(M)$ are Frechet-Schwartz spaces. \\

By definition  the Poincar\'e lemma for $\opa_M$  is said to be valid  at $p_0\in M$ at bidegree $(p,q)$ if and only if the sequence induced by $\opa_M$ on stalks
$$\mathcal{C}^\infty_{p,q-1}\lbrace p_0\rbrace \overset{\opa_M}\longrightarrow \mathcal{C}^\infty_{p,q}\lbrace p_0\rbrace \overset{\opa_M}\longrightarrow \mathcal{C}^\infty_{p,q+1}\lbrace p_0\rbrace  $$
is exact. Here $\mathcal{C}^\infty_{p,q}\lbrace p_0\rbrace $ denotes the stalk at $p_0$ in the sheaf of germs of $\mathcal{C}^\infty$ $(p,q)$-forms over $M$. Hence the Poincar\'e Lemma for $\opa_M$ fails to hold at $p_0$ at bidegree $(p,q)$ if for every sufficiently small open neighborhood $\Omega$ of $p_o$ there exists a smooth $(p,q)$-form $f$ on $\Omega$ with $\opa_M f=0$ in $\Omega$ which is not $\opa_M$-exact {\it on any} open neighborhood $\omega \subset\Omega$ of $p_o$. This is of course a local property of $M$ near $p_0$.\\

Throughout our paper, we also have to use smooth, compactly supported forms, which will be denoted by $\mathcal{D}^{p,q}(M)$. Similarly, $\mathcal{D}^{p,q}_K(M) = \lbrace f\in\mathcal{D}^{p,q}(M)\mid\mathrm{supp} f \subset \ K\rbrace$ for $K\subset M$ compact. Also, $\Vert\ \Vert_{K,m}$ will denote the usual $\mathcal{C}^k$-norm of forms on $K$ (with respect to a choice of a smooth Riemannian metric on $M$ and a smooth partition of unity).\\

In order to better describe $CR$ manifolds geometrically, it is convenient to introduce the
{\it characteristic conormal bundle} of $M$, which we
 denote by $H^o M=\lbrace \xi\in T^\ast M\mid \langle X,\xi\rangle =0, \forall X\in H_{\pi(\xi)}M\rbrace$. Here $\pi: T M \longrightarrow M$ is the natural projection. To each $\xi\in H^o_p M$, we associate the Levi form at $\xi:$
$$\mathcal{L}_p(\xi, X) = \xi(\lbrack J\tilde X, \tilde X\rbrack )= d\tilde\xi(X,JX) \ \mathrm{for} \ X\in H_p M$$
which is Hermitian for the complex structure of $H_p M$ defined by $J$. Here $\tilde \xi$ is a section of $H^o M$ extending $\xi$ and $\tilde X$ a section of $HM$ extending $X$. \\

Finally, a $CR$ manifold $M$ is called {\it pseudoconcave} if  at each point $x\in M$ and every characteristic conormal direction $\xi\in H^o_x(M)\setminus \lbrace 0\rbrace$, the Levi form $\mathcal{L}_x(\xi,\cdot)$ has at least one negative and one positive eigenvalue.\\

\section{Main results}

\begin{thm}\label{first}   
Let $M$ be an abstract 
$CR$ manifold of type $(n,k)$. Assume that there exists a point $p_0\in M$ and a characteristic conormal direction $\xi\in H^o_{p_0}M$ such that the Levi form $\mathcal{L}_{p_0}(\xi,\cdot)$ has $q$ negative and $n-q$ positive eigenvalues.  Then for $0\leq p\leq n+k$, the following holds: Either 
$H^{p,q}(M)$ is infinite dimensional or 
$H^{p,q+1}(M)$ is not Hausdorff and either $H^{p,n-q}(M)$ is infinite dimensional or 
$H^{p,n-q+1}(M)$ is not Hausdorff.
\end{thm}
We use the notation \begin{equation*}
H^{p,q}((p_0)) = \varinjlim_{U\ni{p}_0}
H^{p,q}(U)
\end{equation*}
for the local cohomology groups of $\, \opa_M$, on which we consider the
projective limit topology.

\begin{thm}\label{second}   
Under the same hypothesis 
of Theorem~\ref{first}
we have that, for $0\leq p\leq n+k$, the following holds: Either the local cohomology group $H^{p,q}((p_0))$ 
is infinite dimensional or 
$H^{p,q+1}((p_0))$ is not Hausdorff and either $H^{p,n-q}((p_0))$ is infinite dimensional or 
$H^{p,n-q+1}((p_0))$ is not Hausdorff.\\
In particular the Poincar\'e lemma for $\opa_M$ fails to hold at the point $p_o$ at either bidegree $(p,q)$ or at bidegree $(p,q+1)$ and the Poincar\'e lemma for $\opa_M$  fails to hold at the point $p_0$ at either bidegree $(p,n-q)$ or at bidegree $(p,n-q+1)$.
\end{thm}

\begin{cor}\label{cor}   
Let $M$ be an abstract $CR$ manifold of type $(n,k)$ which is pseudoconcave and not compact. Assume that there exists a point $p_0\in M$ and a characteristic conormal direction $\xi\in H^o_{p_0}M$ such that the Levi form $\mathcal{L}_{p_0}(\xi,\cdot)$ has $n-1$ negative and $1$ positive eigenvalues. Then, for $0\leq p\leq n+k$, we have that $H^{p,n-1}(M)$ is infinite dimensional.
\end{cor}

\begin{cor}
\label{third}   
Let $M$ be an abstract $CR$ manifold of type $(n,k)$ which is pseudoconcave. Assume that there exists a point $p_0\in M$ and a characteristic conormal direction $\xi\in H^o_{p_0}M$ such that the Levi form $\mathcal{L}_{p_0}(\xi,\cdot)$ has $n-1$ negative and $1$ positive eigenvalues. Then, for $0\leq p\leq n+k$, the Poincar\'e lemma for $\opa_M$ fails to hold at the point $p_0$ at  bidegree $(p,n-1)$ .
\end{cor}

\section{Geometric set-up} \label{geosetup}

Our proof of Theorem~\ref{first} relies on a well known construction for $CR$ embedded $CR$ manifolds 
at a point where there exists a characteristic conormal direction such that the associated Levi form 
has exactly $q$ negative and $n-q$ positive eigenvalues. 
For the reader's convenience, we will now sketch this construction. 
For more details, we refer the reader to \cite[p. 389\ ff.]{AFN}.

So let $S\ni 0$ be a piece of a smooth $CR$ submanifold of $\C^{n+k}$, of $CR$ dimension $n$ 
and $CR$ codimension $k$,
 such that $\mathcal{L}_0(\xi,\cdot)$ has $q$ negative and $n-q$ positive eigenvalues 
 for some characteristic conormal direction $\xi$. By a suitable choice of holomorphic coordinates
$$z_1,\ldots, z_n, z_{n+1} = t_1 + is_1,\ldots, z_{n+k} = t_k + is_k,$$
on $\C^{n+k}$, we can assume that
$$ S = \lbrace \rho_1(z) = \ldots = \rho_k(z) = 0\rbrace,$$ 
for defining functions of the form 
$$\rho_j = s_j - h_j(z_1,\ldots, z_n, t_1,\ldots, t_k),$$
with $h_j = O(\vert z\vert^2)$ at $0$. 
Then 
$T^{1,0}_0 S = \C^n = \lbrace z_{n+1} = \ldots, z_{n+k} = 0\rbrace,$ 
and 
the assumption 
on the Levi-form of $S$ at $0$ 
means 
that there are $k$ real numbers $\lambda_1,\ldots, \lambda_k$, 
for which 
the hermitian form 
\begin{equation*}
 {\sum}_{\mu,\nu=1}^n\frac{\pa^2({\sum}_{j=1}^k\lambda_jh_j)}{\partial{z}_\mu\partial\bar{z}_\nu}(0)z_\mu\bar{z}_\nu
\end{equation*}
is nondegenerate on $\mathbb{C}^n$, with $q$ negative and $n-q$ positive eigenvalues.
\par 
Set $h={\sum}_{j=1}^k\lambda_j\rho_j$.
We set $h = \sum_{\alpha = 1}^k \lambda_\alpha h_\alpha$. 
In $\C^n=T^{1,0}_0S$, 
we may assume 
$h$ 
to be in diagonal form, i.e. 
\begin{equation*}
 \left( \frac{\pa^2{h}}{\pa{z}_\mu\pa\bar{z}_\nu}(0)\right)_{1\leq\mu,\nu\leq{n}}=
 \begin{pmatrix}
-\mathrm{I}_q& 0\\
0 &\mathrm{I}_{n-q} 
\end{pmatrix}.
\end{equation*}

Set 
$$\phi = i \left(\sum_{\alpha = 1}^k \lambda_\alpha t_\alpha\right) - h(z,t) +  2\sum_{\mu,\nu 
= 1}^n\frac{\pa^2 h}{\pa z_\mu\pa z_\nu}(0) z_\mu z_\nu  - m\sum_{\alpha = 1}^q \vert z_\alpha\vert^2 - 
m\sum_{\alpha = 1}^k(t_\alpha + ih_\alpha)^2$$
for some suffiently large $m > 0$. Then
$$\mathrm{Re}\,\phi (z) \leq -\frac{1}{2} (\sum_{\alpha = 1}^n 
\vert z_\alpha\vert^2 + \sum_{\alpha = 1}^k t_\alpha^2) \qquad \mathrm{near}\ 0.$$
In fact, after approximating $\mathrm{Re}\,\phi$  
by its second order Taylor
polynomial $\mathrm{Re}\,\phi_2$, 
the remainder is 
$\mathrm{O}(|z|^3+|t|^3)$, and hence
bounded by a small constant times $|z|^2+|t|^2$ on a neighborhood of $0$.
This reduces the proof to prove the estimate with $\mathrm{Re}\,\phi_2$
on the left hand side and $\tfrac{1}{2}$ substituted by any 
$\mathrm{constant}>\tfrac{1}{2}$
on the right hand side. This can be obtained
by using the elementary inequality $2ab\leq ca^2 + c^{-1}b^2$
and  taking a large $m>0$ 
to 
take care of the terms involving  the
second order derivatives $\pa^2h(0)/\partial{t}_\beta\partial{z}_\alpha$
and $\pa^2h(0)/\partial{t}_\beta\partial{\bar{z}}_\alpha$. 

For $\lambda > 0$ we then define the  
\textsl{peak forms}
$$f_\lambda = e^{\lambda\phi} dz_1\wedge\ldots\wedge dz_p\wedge d\ol z_1\wedge \ldots\wedge d\ol z_q. 
$$
They are 
smooth 
$(p,q)$-forms
on $S$ satisfying $\opa_S f_\lambda = 0$ (note that $t_\alpha + ih_\alpha$ 
is the restriction to $S$ of the holomorphic function $z_{n+\alpha}$, $\alpha = 1,\ldots, k$).

Similarly we set
$$\psi = -i \left(\sum_{\alpha = 1}^k \lambda_\alpha t_\alpha\right) + h(z,t) - 2\sum_{\mu,\nu = 1}^n\frac{\pa^2 h}{\pa z_\mu\pa z_\nu}(0) z_\mu z_\nu - m \sum_{\alpha = q+1}^n \vert z_\alpha\vert^2 - m\sum_{\alpha = 1}^k(t_\alpha + ih_\alpha)^2$$

for some suffiently large $m > 0$. Then
$$\mathrm{Re}\,
\psi (z) \leq -\frac{1}{2} (\sum_{\alpha = 1}^n \vert z_\alpha\vert^2 
+ \sum_{\alpha = 1}^k t_\alpha^2) \qquad \mathrm{near}\ 0,$$
and we define 
another one-parameter family of \textsl{peak forms}, which are 
of degree $(n+k-p, n-q)$, on~$S$:
$$g_\lambda = e^{\lambda\psi} dz_{p+1}\wedge\ldots\wedge dz_{n+k}\wedge d\ol z_{q+1}\wedge\ldots\wedge d\ol z_n.$$

Again we have $\opa_S g_\lambda =0$.\\

In the proof of Theorem \ref{first}, \ref{second}, \ref{ffirst}, \ref{ssecond} 
the forms $f_\lambda$ and $g_\lambda$ play an essential role, 
because their properties 
will be used to 
contradict  
certain a priori estimates related to the validity of the Poincar\'e lemma.
The proofs of our theorems 
rely 
on 
constructing suitable forms
that agree to infinite order, 
at some points, with the pullbacks of $f_\lambda,g_\lambda$.
\par

\section{Proofs}
\begin{proof}[Proof of Theorem~\ref{first}]
Let us set
\begin{align*}
\mathcal{Z}^{p,q}(M)&=\{f\in\mathcal{C}^\infty_{p,q}(M)\mid
\opa_M{f}=0\},\\
\mathcal{Z}^{p,q}_0(M)&=\{f\in\mathcal{Z}^{p,q}(M)\mid
[f]=0\},
\end{align*}
where $[f]$ is the cohomology class of 
$f\in\mathcal{Z}^{p,q}(M)$
in $H^{p,q}(M)$. \par 
The map $f\to[f]$ is continuous. Thus,
if we assume that $H^{p,q}(M)$ is Hausdorff, then the subspace
$\mathcal{Z}_0^{p,q}(M)$  is closed and hence Fr{\'e}chet.
  As a consequence of the open mapping theorem 
  we also get an a priori estimate: 
   For every compact $K_{p,q-1}\Subset{M}$ and integer $m_{p,q-1}\geq{0}$ 
   there is a compact 
   $K_{p,q}\Subset{M}$, an integer $m_{p,q} \geq 0$ and a constant 
   $C_{p,q}>0$ such that 
\begin{equation}\label{51}\begin{cases}
\forall f\in\mathcal{Z}_0^{p,q}(M),\;\\
\exists\, {u}\in\mathcal{C}^\infty_{p,q-1}(M),\end{cases}
\;\;\text{such that}\; \;\begin{cases} \opa_Mu=f,\\[5pt]
\|u\|_{K_{p,q-1},m_{p,q-1}}\leq{C}_{p.q}\|f\|_{K_{p.q},m_{p,q}}.
\end{cases}
\end{equation}

Using Stokes' formula and \eqref{51}, we obtain 
\begin{equation}\label{apriori3}\left\{
\begin{aligned}
\left|\int_K
f\wedge{g}\,\right| \leq C_{p,q}\cdot
\|f\|_{K_{p,q},m_{p,q}}\cdot\|\opa_Mg\|_{K,0},
\quad\qquad\qquad  \\
\forall f\in\mathcal{Z}_0^{p,q}(M),\;
\forall 
g\in\mathcal{D}_K^{n+k-p,n-q}(M),
\end{aligned}\right.
\end{equation}
where $K$ is a compact subset contained in an oriented open submanifold  of $M$ and
$K_{p,q-1}\supset{K}$. \par 
To prove Theorem~\ref{first}, we argue by contradiction, assuming that
the dimension $\ell$ of 
$H^{p,q}(M)$ is finite and that 
$H^{p,q+1}(M)$ is Hausdorff. In particular, also
$H^{p,q}(M)$ is Hausdorff and 
\eqref{51} holds for both
$(p,q)$ and  $(p,q+1)$. \par 

Let  
$V$ be an oriented 
open neighborhood of $p_0\in M$ such that for every point $x\in V$, 
there exists a characteristic conormal direction $\xi_x$ such that 
$\mathcal{L}_x(\xi_x,\cdot)$ has $q$ negative and $n-q$ positive eigenvalues.\par
 
Fix 
$\ell$ 
distinct
points $p_1,\ldots, p_\ell$ 
in 
$V,$ 
all different from $p_0$ . 
Later on we shall choose 
cut-off functions $\chi_j,\ j=0,1,\ldots, \ell$, having 
disjoint 
compact supports 
in 
sufficiently small 
neighborhoods 
of each $p_j$,  
and 
such that $\chi_j\equiv 1$ near $p_j$.
The compact $K\Subset{V}$ will be the union of $\mathrm{supp}\,{\chi_j}$. 
For each $0\leq j\leq \ell$, we 
make the following construction:\par
\smallskip

Having fixed smooth coordinates centered at $p_j$,  by 
the formal Cauchy-Kowalewski procedure of \cite{AH1,AFN}, 
we 
find smooth complex valued functions 
$\varphi = (\varphi_1,\ldots, \varphi_{n+k})$ 
in an open neighborhood $V_j$ of $0$ with each $\varphi_i(0)=0,\; 
d\varphi_1\wedge\ldots\wedge\varphi_{n+k}\not= 0$ in $V_j$, 
and 
$\opa_M \varphi_i$ 
vanishing
to infinite order at $0$. 
Then $\varphi: V_j\longrightarrow \mathbb{C}^{n+k}$ 
gives a smooth local embedding 
$\tilde{M}_j = \varphi(V_j)$ of $M$ 
into $\mathbb{C}^{n+k}$. 
The $CR$ structure on 
$\tilde{M}_j$ 
induced from $\mathbb{C}^{n+k}$
agrees to infinite order at $0$ with the original 
one 
on $M$ at $p_j$.
In particular 
$\tilde{M}_j$
is a smooth real submanifold 
in $\mathbb{C}^{n+k}$ sitting inside a strictly $(n-q+k-1)$-pseudoconvex 
and strictly $q$-pseudoconcave real hypersurface (this means that this
hypersurface has a real valued smooth
defining function whose complex Hessian has 
signature
$(n-q+k-1,\,q)$ when restricted to
its analytic tangent). 
Thus,
after possibly shrinking $V_{\! j},$
we can find 
smooth complex valued functions $\phi_j$ and $\psi_j$ 
on $V_{\!{j}},$ with $\opa_M\phi_j$ and $\opa_M\psi_j$ vanishing to infinite order 
at $0$ 
and, by fixing $m>2$ in \S{4},  
satisfying
\begin{align}  \label{phi}
\mathrm{Re}\,\phi_j &\leq -\frac{1}{2}\vert x\vert^2 \quad\mathrm{on} \ V_{\!{j}},
\\ \label{psi}
\mathrm{Re}\,\psi_j &\leq -\frac{1}{2}\vert x\vert^2 \quad\mathrm{on} \ V_{\!{j}},
\\
  \label{phi+psi}
\phi_j +\psi_j & = -2\vert x\vert^2 + O(\vert x\vert^3)\quad\mathrm{on} \ V_{\!{j}}
\end{align}
for 
the 
coordinate chart $x$ centered at $p_j$ 
(they are the pullbacks by $\varphi$ of the $\phi,\,\psi$ on $\tilde{M}_j$ of \S\ref{geosetup}).
\par

Moreover, 
by  
a suitable choice of holomorphic coordinates in
$\C^{n+k}$, 
we obtain that 
$T^\ast M$ is spanned near $p_j$ 
by forms
\begin{align*}
 \omega_1 = dz_1 + O(\vert x\vert^\infty),\; \hdots,\;
 \omega_n = dz_n +
 O(\vert x\vert^\infty),\;\hdots,\;
 \ol\omega_1= d\ol z_1 + O(\vert x\vert^\infty),\qquad \\
 \hdots,\; \ol\omega_n  = d\ol z_n = O(\vert x\vert^\infty),\;
 \theta_1 = dx_{2n+1} + O(\vert x\vert^\infty),\;\hdots, \;
\theta_k = dx_{2n+k} +  O(\vert x\vert^\infty),
\end{align*}
 which are $d$-closed to infinite order at $0$, 
 and 
 $
 {T}^{1,0}M$ is spanned by $\omega_1,\hdots,\omega_n$ and 
 $
 T^{0,1}M$ 
 by $\ol\omega_1,\ldots,\ol\omega_n$ 
 on a neighborhood of $p_j$. 
 Following again \cite{AFN} or \cite{HN2}, 
 by the geometric condition on the Levi-form 
 at $p_j$ 
 we may also assume 
 that $\opa_M(\phi_j \wedge\ol\omega_1\wedge\ldots\wedge\ol\omega_q)$ 
 and $\opa_M(\psi_j\wedge\ol\omega_{q+1}\wedge\ldots\wedge\ol\omega_n)$
  vanish to infinite order at $p_j$. 
\par

For each real $\lambda > 0$ we now define 
the smooth $(p,q)$-form
$$ f_j^\lambda = \chi_{\!{j}} e^{\lambda\phi_j} \omega_1\wedge\ldots\wedge\omega_p
\wedge\ol\omega_1\ldots\wedge\ol\omega_q,$$
where the cut-off function $\chi_{\!{j}}$ has compact support contained in $V_j$. 
Moreover 
our choice  
of $\phi_j$ 
implies 
that $\opa_M(f_j^\lambda)$ is rapidly decreasing 
with respect to 
$\lambda$ 
in the topology of $\mathcal{D}_K^{p,q+1}(M)$, 
as $\lambda$ tends to infinity. 
Indeed, by (\ref{phi}) the function $\opa_M[\exp(\lambda\phi_j)],$ and any derivative of it with respect to $x$, 
is rapidly decreasing as $\lambda\rightarrow +\infty$ in any fixed small neighborhood of $p_j$.
Indeed, in a suitable trivialization, the components of any derivative of $\opa_M[\exp(\lambda\phi_j)]$ are bounded
on a neighborhood $U_0$ 
of $x=0$ by $h(x)\exp(-\lambda|x|^2/2)$, for a positive function $h$ which vanishes to
infinite order at $0$. For any compact subset $\kappa\Subset{U}_0$ and any integer $m>{0}$,
we obtain 
\begin{equation*}
 h(x)\exp(-\lambda|x|^2/2)\leq c_m|x|^m\exp(-\lambda|x|^2/2)\leq c_m(m/\lambda)^{m/2}\exp(-m/2),\;\;
 \forall x\in\kappa,
\end{equation*}
with a constant $c_m>0$ independent of $\lambda$.
The terms containing a derivative of $\chi_j$ are rapidly decreasing in virtue of \eqref{phi},
because they have support
in an annulus $\{0<r'\leq |x|\leq r''\}$.
\par 

We also set
$$ g_j^\lambda = \chi_j e^{\lambda\psi_j} \omega_{p+1} \wedge\ldots\wedge\omega_n\wedge\theta_1\wedge\ldots
\wedge\theta_k\wedge\ol\omega_{q+1}\wedge\ldots\wedge\ol\omega_n.$$
Then, arguing as before, 
we get that also 
$\opa_M(g_j^\lambda)$ is rapidly decreasing with respect to 
$\lambda$ in the topology of $\mathcal{D}_K^{n+k-p,n-q}(M)$, as $\lambda$ tends to infinity.\par

Next, 
using \eqref{51},
we solve $\opa_M u_j^\lambda = \opa_M f_j^\lambda$ with an estimate
\begin{equation}  \label{apriori4}
\Vert u_j^\lambda\Vert_{K_{p,q},m_{p,q}} \leq 
C_{p,q+1} \Vert \opa_M f_j^\lambda\Vert_{K_{p,q+1},m_{p,q+1}} . 
\end{equation} 
Hence 
$\Vert u_j^\lambda\Vert_{K_{p,q},m_{p,q}}$ is rapidly decreasing with respect to 
$\lambda$.
 The forms 
 $\tilde f_j^\lambda = f_j^\lambda - u_j^\lambda$ 
 are $\opa_M$-closed 
 on $M$. \par
 
 Since $\mathrm{dim}_\C H^{p,q}(M)=\ell$, there 
 are constants $c_0^\lambda,\ldots, c_\ell^\lambda$, not all equal to zero, such that
 $$c_0^\lambda\tilde f_0^\lambda + \ldots + c_\ell^\lambda \tilde f_\ell^\lambda \in\mathcal{Z}_0^{p,q}(M). $$
 To get a contradiction, we are going to use the estimate (\ref{apriori3}) with 
 $f=\sum_{j=0}^\ell c_j^\lambda \tilde f_j^\lambda$ and 
 $g= \sum_{j=0}^\ell \ol c_j^\lambda g_j^\lambda$. We have

\begin{align} \label{fwedgeg}  
 \int_K{f}\wedge g  & = 
\int_K
 \big(\sum_{j=0}^\ell c_j^\lambda \tilde f_j^\lambda\big)
 \wedge \big(\sum_{j=0}^\ell \ol c_j^\lambda g_j^\lambda\big) 
 \\ 
 & =  
 \int_K\big(\sum_{j=0}^\ell c_j^\lambda ( f_j^\lambda-u_j^\lambda)\big)
 \wedge \big(\sum_{j=0}^\ell \ol c_j^\lambda g_j^\lambda\big)\nonumber\\
& =  \sum_{j=0}^\ell\vert c_j^\lambda\vert^2  
\int_K{f}_j^\lambda\wedge g_j^\lambda - 
\int_K\sum_{i,j=0}^\ell c_i^\lambda \ol c_j^\lambda u_i^\lambda\wedge g_j^\lambda . \notag 
\end{align}
Note that, while writing the first sum in the last equality,
we 
used that the $\chi_j$'s have disjoint supports.\par

We are now going to estimate the term on the right of (\ref{fwedgeg}). We have

$$ 
\int_K
f_j^\lambda\wedge g_j^\lambda  =   
\int_K\chi_j^2 e^{\lambda(\phi_j + \psi_j)} 
\omega_1\wedge\ldots\wedge\omega_n\wedge\theta_1\wedge\ldots\theta_k
\wedge\ol\omega_1\wedge\ldots\wedge\ol\omega_n $$
$$=  
\int_K\lbrace \chi_j^2 e^{\lambda(-2\vert x\vert^2 + O(\vert x\vert^3)} 
+ O(\vert x\vert)\rbrace dz_1\wedge\ldots\wedge dz_n
\wedge d\ol z_1\wedge\ldots\wedge d\ol z_n\wedge dx_{2n+1}\wedge\ldots\wedge dx_{2n+k}.$$
  
Making the change of variables $y=\sqrt\lambda\, x$, 
and afterwards changing the name of $y$ back to $x$,  we get

\begin{align*} 
& 
\int_K
f_j^\lambda\wedge g_j^\lambda \\
& = \lambda^{-n-\frac{k}{2}} \Big\lbrace
\int_K
 \chi_j^2\left(
\frac{x} {\sqrt\lambda}\right)  & e^{-2\vert x\vert^2 
+ O(\lambda^{-\frac{1}{2}})}dz_1\wedge \ldots\wedge d\ol z_n\wedge dx_{2n+1}\wedge\ldots\wedge dx_{2n+k}\\  
 &&  + O(\lambda^{-\frac{1}{2}})\Big\rbrace.
\end{align*}

Therefore we obtain
\begin{equation} \label{firstintegral}
\left\vert 
\int_K
f_j^\lambda\wedge g_j^\lambda\right\vert 
\geq c \lambda^{-n-\frac{k}{2}}
\end{equation}
for some constant $c > 0$.\\

Also we can use (\ref{apriori4}) to get
\begin{eqnarray*}
\left\vert 
\int_K
\sum_{i,j=0}^\ell c_i^\lambda \ol c_j^\lambda u_i^\lambda\wedge g_j^\lambda \right\vert & \lesssim & 
\sum_{j=0}^\ell\vert c_j^\lambda\vert^2 \sup_{i,j} (\Vert u_i^\lambda\Vert_{K,0}\cdot\Vert g_j^\lambda\Vert_{K,0})\\
& \lesssim & \sum_{j=0}^\ell\vert c_j^\lambda\vert^2 \sup_{i,j} 
(\Vert \opa_M f_i^\lambda\Vert_{K_{p,q+1},m_{p,q+1}}\cdot\Vert g_j^\lambda\Vert_{K,0}).
\end{eqnarray*}

Now $\Vert\opa_M f_i^\lambda\Vert_{K_{p,q+1},m_{p,q+1}}$ is rapidly decreasing with respect to $\lambda$, 
whereas $\Vert g_j^\lambda\Vert_{K,0}$ is 
at most 
of polynomial growth with respect to $\lambda$, hence we get 
\begin{equation}  \nonumber
\left\vert 
\int_K
\sum_{i,j=0}^\ell c_i^\lambda \ol c_j^\lambda u_i^\lambda\wedge g_j^\lambda \,\right\vert \leq 
\sum_{j=0}^\ell\vert c_j^\lambda\vert^2 \lambda^{-n-k}
\end{equation}
for sufficiently large $\lambda$. Combining this with (\ref{firstintegral}), we get
\begin{equation} \label{integral}
\left\vert  
\int_K
f\wedge g\,\right\vert \geq \frac{c}{2}\sum_{j=0}^\ell\vert c_j^\lambda\vert^2  \lambda^{-n-\frac{k}{2}}
\end{equation}  
for sufficiently large $\lambda$.

On the other hand,
using (\ref{apriori3}), we can estimate $ 
\int_K
f\wedge g$ as follows:

$$\left\vert 
\int_K
f\wedge g\,\right\vert  \leq  C_{p,q} \Vert f\Vert_{K_{p,q},m_{p,q}} \cdot \Vert \opa_M g\Vert_{K,0}$$
$$ \lesssim  \sum_{j=0}^\ell \vert c_j^\lambda\vert^2 \sup_{i,j} (\Vert\tilde f_j^\lambda\Vert_{K_{p,q},m_{p,q}} 
\cdot \Vert \opa_M g_j^\lambda\Vert_{K,0} ) $$
$$ \lesssim  \sum_{j=0}^\ell \vert c_j^\lambda\vert^2 \sup_{i,j} (\Vert  f_j^\lambda\Vert_{K_{p,q+1},m_{p,q+1}+1} 
\cdot \Vert \opa_M g_j^\lambda\Vert_{K,0}).$$

Since $\Vert f_j^\lambda\Vert_{K_{p,q+1},m_{p,q+1}+1}$ 
has at most 
polynomial growth, 
whereas $\Vert \opa_M g_j^\lambda\Vert_{K,0}$ 
is rapidly decreasing with respect to $\lambda$, we get that
$$\left\vert 
\int_K
 f\wedge g\,\right\vert  \lesssim \sum_{j=0}^\ell \vert c_j^\lambda\vert^2 \lambda^{-n-k}.$$
This contradicts (\ref{integral}) and therefore proves that either
 $H^{p,q}(M)$  has to be infinite dimensional or $H^{p,q+1}(M)$ has to be
not Hausdorff. \\

Now, replacing $\xi$ by $-\xi$,  
and $q$ by $n\! - \! q$, 
it also follows that either $H^{p,n-q}(M)$  is infinite dimensional or  $H^{p,n-q+1}(M)$ is not Hausdorff.\\

For $q=0$, the statement was proved in \cite{BHN} and is similar 
to Boutet de Monvel's result \cite{BdM}: In this case, 
the $\tilde{M}_j$'s are 
contained 
in 
strictly pseudoconvex real  
hypersurfaces. 
If $H^{p,1}(M)$ was Hausdorff, then in particular the range of 
$\opa_M$ would be closed in $\mathcal{C}^\infty_{p,1}(M)$,  
and 
one could construct infnitely many linearly independent $CR$ functions on $M$ 
as in \cite{BHN}.

Also, the Levi-form $\mathcal{L}_{p_o}(-\xi,\cdot)$ has $n> 0$ negative and $0$ positive eigenvalues. 
By what was already proved, we therefore know that $H^{p,n}(M)$ is infinite dimensional, as 
in this case the closed range condition is trivially fulfilled   
(note that $H^{p,n+1}(M)$ is always zero).  \end{proof}

\begin{proof}[Proof of Theorem~\ref{second}]  
The proof is essentially the same as the one of Theorem \ref{first} 
and follows \cite{AFN} for the additional functional 
analysis
arguments involved. \par 
For an open neighborhood $\omega$ of $p_0$ we set 
\begin{align*}
\mathcal{Z}^{p,q}(\omega)&=\{f\in\mathcal{C}^\infty_{p,q}(\omega)\mid
\opa_M{f}=0\},\\
\mathcal{Z}^{p,q}_0(\omega)&=\{f\in\mathcal{Z}^{p,q}(\omega)\mid
[f]_{p_0}=0\},
\end{align*}
where $[f]_{p_0}$ is the local cohomology class of 
$f\in\mathcal{Z}^{p,q}(\omega)$
in $H^{p,q}(p_0)$. \par 
The map $f\to[f]_{p_0}$ is continuous. Thus,
if we assume that $H^{p,q}(p_0)$ is Hausdorff, then, for every
open neighborhood $\omega$ of $p_0$ in $M$, the subspace
$\mathcal{Z}_0^{p,q}(\omega)$  is closed and hence Fr{\'e}chet.  
By using 
Baire's category theorem 
we   
show that, for every open neighborhood 
$\omega$ of $p_0$ in $M$, we can find an open neighborhood $\omega_0$ 
 of $p_0$ in 
 $\omega$ 
 with the property  
 that, for all 
 $f\in\mathcal{Z}^{p,q}_0(\omega)$, 
 there is  
 a solution 
 $u\in\mathcal{C}^\infty_{p,q-1}(\omega_0)$ to $\,\opa_M{u}=f|_{\omega_0}$.
By   
  the open mapping theorem for Fr{\'e}chet
   spaces we also get an a priori estimate: 
   For every compact $K\Subset\omega_0$ there 
   are
   a compact 
   $K_1\Subset \omega$, an integer $m_1 \geq 0$ and a constant 
   $C_1>0$ such that  
   a 
   solution $u$ to $\,\opa_M u = f$ 
   can be chosen to satisfy
$$\Vert u\Vert_{K, 0} \leq C_1\Vert f\Vert_{K_1,m_1}.$$
This a priori estimate is analogous to \eqref{51}. 
\par

As before, we get 
a 
crucial estimate similar to (\ref{apriori3}):
\begin{equation*}
\left| 
\int_K
f\wedge{g}\,\right| \lesssim \|f\|_{K_1,m_1}\|\opa_Mg\|_{K,0},
\quad\forall f\in\mathcal{Z}^{p,q}_0(\omega),\;\forall
g\in\mathcal{D}^{n+k-p, n-q}_{K}(M).
\end{equation*}

The rest of the proof now follows the proof of Theorem \ref{first}.\end{proof}

\begin{proof}[Proof of Corollary \ref{cor}]
The statement of the Corollary follows from Theorem~\ref{first}, 
together with Malgrange's vanishing theorem for pseudoconcave $CR$ manifolds proved in \cite{BH1}: 
\textsl{Let $M$ be an abstract $CR$ manifold that is pseudoconcave and not compact. 
Then $H^{p,n}(M)=0$ for $0\leq p\leq n+k$.}
\end{proof}

\begin{proof}[Proof of Corollary \ref{third}]
The statement of the Corollary immediately follows from Theorem~\ref{second}, 
together with the validity of the Poincar\'e lemma for top-degree forms 
on pseudoconcave abstract $CR$ manifolds, which was proved in \cite{B}.
\end{proof}

\section{The case of currents}
For $U$ open in $M$ we
denote by $\Hdis^{p,q}(U)$  the cohomology groups of $\bar{\partial}_M$ 
on distribution sections. The inclusion $\Ci_{p,q}(U)\subset\Dis_{p,q}(U)$ yields
a map ${H}^{p,q}(U)\to\Hdis^{p,q}(U)$. 
Let us set 
\begin{equation}
 \mathcal{Z}^{p,q}_{w}(U)=\{f\in\mathcal{Z}^{p,q}(U)\mid f\sim{0}\;\;\text{in}\;\;\Hdis^{p,q}(U)\}.
\end{equation}
For the quotients 
\begin{equation}
 \tilde{H}^{p,q}(U)=\mathcal{Z}^{p,q}(U)/\mathcal{Z}_w^{p,q}(U)
\end{equation}
we have  natural maps 
\begin{equation}
H^{p,q}(U)\twoheadrightarrow\tilde{H}^{p,q}(U)\hookrightarrow\Hdis^{p,q}(U),
\end{equation}
the first one being onto, the second one injective, and both being continuous for the quotient topologies. 
In particular, $\tilde{H}^{p,q}(U)$ is Hausdorff when $\Hdis^{p,q}(U)$ is Hausdorff. 
We shall prove the following generalization of Theorem~\ref{first}.
\begin{thm}\label{ffirst}   
Let $M$ be an abstract 
$CR$ manifold of type $(n,k)$. 
Assume that there exists a point $p_0\in M$ and a characteristic conormal direction 
$\xi\in H^o_{p_0}M$ such that the Levi form $\mathcal{L}_{p_0}(\xi,\cdot)$ 
has $q$ negative and $n-q$ positive eigenvalues.  
Then for $0\leq p\leq n+k$, the following holds: Either 
$\tilde{H}^{p,q}(M)$ is infinite dimensional or 
$\tilde{H}^{p,q+1}(M)$ is not Hausdorff and either $\tilde{H}^{p,n-q}(M)$ is infinite dimensional or 
$\tilde{H}^{p,n-q+1}(M)$ is not Hausdorff
\end{thm}
We use the notation \begin{equation*}
\tilde{H}^{p,q}((p_0)) = \varinjlim_{U\ni{p}_0}\tilde{H}^{p,q}(U).
\end{equation*}
for the local cohomology groups of $\, \opa_M$, on which we consider the
projective limit topology.  
\begin{thm}\label{ssecond}   
Under the same hypothesis  
of Theorem~\ref{first}
we have that, for $0\leq p\leq n+k$, the following holds: Either the local cohomology group 
$\tilde{H}^{p,q}((p_0))$ 
is infinite dimensional or 
$\tilde{H}^{p,q+1}((p_0))$ is not Hausdorff and either $\tilde{H}^{p,n-q}((p_0))$ is infinite dimensional or 
$\tilde{H}^{p,n-q+1}((p_0))$ is not Hausdorff.\par
In particular the Poincar\'e lemma for $\opa_M$ on distribution sections
fails to hold at the point $p_0$ at either bidegree $(p,q)$ or at bidegree $(p,q+1)$ and the Poincar\'e lemma for $\opa_M$  on distribution sections
fails to hold at the point $p_0$ at either bidegree $(p,n-q)$ or at bidegree $(p,n-q+1)$.
\end{thm}
\begin{proof}[Proof of Theorem~\ref{ffirst}]
The main ingredient in the proof of Theorems~\ref{ffirst},\ref{ssecond}
is to substitute the a priori estimate 
\eqref{apriori3} with an a priori estimate of the form \begin{equation}
\label{apriori7}
\left\{
\begin{aligned}
\left| \int_K f\wedge{g}\right| \leq C_{p,q}\|f\|_{K_{p,q},m_{p,q}}\cdot 
\|\bar{\partial}_Mg\|_{K,\nu_{p,q}},\qquad\qquad\qquad\\
\forall f\in\mathcal{Z}^{p,q}_w(M),\;\forall g\in\mathcal{D}_K^{n+k-p,n-q}(M).
\end{aligned}\right.
\end{equation}
Here $K, \; K_{p,q}$ are compact sets in $M$, with $K$ contained in an oriented open submanifold of $M$, 
and $m_{p,q},\nu_{p,q}$ non negative
integers. \par 
Fix a compact $K\Subset{M}$ and a relatively compact oriented open neighborhood $U$ of
$K$ in $M$. Given a Riemannian metric on $M$ we can define the Sobolev spaces
with negative exponents $W^{-\ell}_{p,q}(U)$. Since the restriction of
a distribution to a relatively compact open subset has finite order, we obtain
\begin{align*}
&\mathcal{Z}_w^{p,q}(M)={\bigcup}_{\nu=0}^\infty\piup_1(E_\nu),\;\; \text{where}
\\
& \qquad E_\nu=\{(f,u)
\in \mathcal{Z}_w^{p,q}(M)\times\Dis^{p,q-1}(M)\mid \bar{\partial}u=f,\;
u|_U\in{W}^{-\nu}_{p,q-1}(U)\},
\end{align*}
and $\piup_1$ is the projection on the first component.
Since $\mathcal{Z}_w^{p,q}(M)$ is Fr{\'e}chet, there is a $\nu=\nu_{p,q}$
for which  $\piup_1(E_\nu)$ is of the second Baire category. 
The space \begin{equation*}
F_\nu=\{(f,u)
\in \mathcal{Z}_w^{p,q}(M)\times{W}^{-\nu}_{p,q-1}(U)\mid \bar{\partial}u=f|_U\}
\end{equation*}
is a Fr{\'e}chet subspace of the product 
$\mathcal{Z}_w^{p,q}(M)\times{W}^{-\nu}_{p,q-1}(U)$
and then the fact that $\piup_1(F_\nu)$ contains a 
$\piup_1(E_\nu)$ which is of the second Baire category
implies that $\piup_1(F_\nu)=\mathcal{Z}_w^{p,q}(M)$. By the open mapping theorem,
for $\nu=\nu_{p,q}$ we can find a compact $K_{p,q}\Subset{M}$, 
an integer $m_{p,q}\geq{0}$ and a constant $C_{p,q}'>0$ such that 
\begin{equation}\label{65}
\forall f\in\mathcal{Z}^{p,q}_w(M),\;\exists u\in{W}^{-\nu_{p,q-1}}_{p,q}(U)
\;\text{s.t.}\;
\begin{cases}
\bar{\partial}_Mu=f|_U, \\
 \|u\|_{{W}^{-\nu}_{p,q-1}(U)}
\leq C_{p,q}'\|f\|_{K_{p,q},m_{p,q}}.
\end{cases}
\end{equation}
Clearly we obtain \eqref{apriori7} from \eqref{65}
and get therefore the proof  by repeating the argument in the
proof of Theorem~\ref{first}.
\end{proof}
\begin{proof}[Proof of Theorem~\ref{ssecond}] 
As before, we need to reduce to an a priori estimate of the form 
\begin{equation*}
\left|\int_K f\wedge{g}\,\right| \lesssim \|f\|_{K_1,m_1}\|\opa_Mg\|_{K,m_2},
\quad\forall f\in\mathcal{Z}^{p,q}_w(\omega),\;\forall
g\in\mathcal{D}^{n+k-p, n-q}_{K}(M),
\end{equation*}
where $\omega$ is an open oriented neighborhood of $p_0$ in $M$ 
and $K,K_1$ compact subsets of $\omega$. This can be done by
using again Baire's category argument, since \begin{align*}
&\mathcal{Z}^{p,q}_w(\omega)={\bigcup}_{\nu}\piup_1(E_\nu)\quad\text{for}\\
&\qquad E_{\nu}=\{(f,u)\in\mathcal{Z}^{p,q}_w(\omega)\times
{W}^{-\nu}_{p,q-1}(\omega_\nu)\mid \bar{\partial}_Mu=f|_{\omega_\nu}\},
\end{align*}
where $\{\omega_\nu\}$ is a fundamental system of open neighborhoods of
$p_0$ which are relatively compact in $\omega$ and $\piup_1$ is projection
on the first factor. The conclusion follows as in the proof of Theorem~\ref{second},
the only difference being to deal with the norm of the $m$-th derivatives
of $\bar{\partial}_Mg$ instead of simply the sup-norm.
\end{proof}

\section{Examples, remarks etc.}
Let $n$ be an integer $\geq{2}$. The Hermitian symmetric $n\times{n}$ matrices
form an $n^2$-dimensional  real linear space $\mathpzc{P}(n)$. Fix a basis
 $H_1,\hdots,H_{n^2}$ of $\mathpzc{P}(n)$. \par
Consider the $CR$ submanifold
$M$ of $\C^{n(n+1)}$ which is defined by 
\begin{equation}
 M=\{(z,w)\in\C^{n^2}\times\C^n\mid \im{z}_j=w^*H_jw+|z|^2,\; 1\leq{j}\leq{n}^2\}.
\end{equation} 
This $M$ is of type $(n,n^2)$. Its scalar Levi forms, corresponding to nonzero characteristics,  
are all non zero, 
and there are non degenerate scalar
Levi forms of all signatures $(q,n-q)$, for $0\leq{q}\leq{n}$. By \cite{AFN}, all local cohomology groups
$H^{p,q}((p_0)),\; \tilde{H}^{p,q}((p_0))$, 
for $p_0\in{M}$, are infinite dimensional, and the argument in the proof of Theorems~\ref{first},\ref{ffirst}
shows that also the global groups $H^{p,q}(M),\; \tilde{H}^{p,q}(M)$ are infinite dimensional, for
$0\leq{q}\leq{n}$, and all $0\leq{p}\leq{n}(n+1)$.
\par \smallskip. 

It is more difficult to produce examples of abstract non-embeddable $CR$ manifolds of higher 
$CR$-codimension, as little is known in this case (see e.g. \cite{HN0}). We sketch a possible example,
which is a variation of the example above.
Let $n$ be an integer $\geq{2}$. Traceless Hermitian symmetric $n\times{n}$ matrices
form an $n^2\! - \!1$-dimensional  real linear space $\mathpzc{P}_0(n)$. Set $k=n^2\! - \!1$
and select a basis $H_1,\hdots,H_k$ of $\mathpzc{P}_0(n)$. \par
We consider the $CR$ submanifold
$M$ of $\C^{n+k}$ which is defined by 
\begin{equation}
 M=\{(z,w)\in\C^k\times\C^n\mid \im{z}_j=w^*H_jw+|z|^2,\; 1\leq{j}\leq{k}\}.
\end{equation} 
This $M$ is of type $(n,k)$. Its scalar Levi forms corresponding to nonzero characteristics 
are all non zero and each has at least 
one positive and one negative eigenvalue, so that $M$ gives an example of a one-pseudoconcave
$CR$ manifold of  high $CR$-codimension.  On the other hand, there are non degenerate scalar
Levi forms of all signatures $(q,n-q)$, for $0<{q}<n$. 
\par 
Our $M$  is contained in an affine real quadric $S$ of $\C^{n+k}$, 
which is a $CR$ hypersurface 
with a non degenerate
Levi form of signature $(1,n+k-2)$. In \cite[\S 6.8]{HN3} it was shown that the closure $\bar{S}$ of $S$ in 
$\mathbb{CP}^{n+k}$, which is a smooth pseudoconcave compact $CR$ hypersurface, admits a
global perturbation of its $CR$ structure that is not locally $CR$-embeddable along the points of a hyperplane
section $D$. This section intersects $M$, because $M$ is pseudoconcave (cf. \cite{HN4}), 
and along its points the new
$CR$ structure agrees to the second  order with the original one. In particular, the 
scalar Levi forms of $M$, for the new $CR$-structure,  do not
change at these points. This provides an $M$ which we guess cannot be embedded into a complex
manifold and has scalar Levi forms of signatures
$(q,n-q)$ for all $0<q<n$. Our results on the global cohomology groups of $M$ apply to the local and 
global cohomology groups $H^{p,q}((p_0)),\; \tilde{H}^{p,q}((p_0))$ (for $p_0\in{M}\cap{D}$), 
$H^{p,q}(M),\;\tilde{H}^{p,q}(M)$ for all $p$ and $1\leq{q}\leq{n}-1$.

\end{document}